\documentclass[a4paper,12pt]{article}
\usepackage{a4wide}
\usepackage{amsmath}
\usepackage{amssymb}
\usepackage{amsthm}
\usepackage{latexsym}
\usepackage{graphicx}
\usepackage[english]{babel}
\usepackage{makeidx}
              
\newtheorem{dfn} [subsection]{Definition}
\newtheorem{obs} [subsection]{Remark}

\newtheorem{prop}[subsection]{Proposition}

\newtheorem{teor}[subsection]{Theorem}
\newtheorem{lema}[subsection]{Lemma}
\newtheorem{cor} [subsection]{Corollary}

\newcommand{\Zng}{finitely generated $\mathbb Z^n$-graded $S$-module }

\begin{document}

\selectlanguage{english}
\frenchspacing

\large
\begin{center}
\textbf{Multigraded modules of Borel type.}

Mircea Cimpoea\c s
\end{center}

\normalsize

\begin{abstract}
In this paper, we introduce the multigraded modules of Borel type and extend several results from the theory of ideals of Borel type. We prove that modules of Borel type are sequentially Cohen Macaulay and pretty clean. Also, we give a formula for the regularity of modules of Borel type.

\vspace{5 pt} \noindent \textbf{Keywords:} Borel type ideals, Mumford-Castelnuovo regularity.

\vspace{5 pt} \noindent \textbf{2000 Mathematics Subject
Classification:}Primary: 13P10, Secondary: 13E10.
\end{abstract}

\section*{Introduction}

Let $K$ be an infinite field, and let $S=K[x_1,\ldots,x_n],n\geq 2$ the polynomial ring over $K$.
Bayer and Stillman \cite{BS} note that Borel fixed ideals $I\subset S$ satisfy the following property:
\[(*)\;\;\;\;(I:x_j^\infty)=(I:(x_1,\ldots,x_j)^\infty)\;for\; all\; j=1,\ldots,n.\] Herzog, Popescu and Vladoiu \cite{hpv} define a monomial ideal $I$ to be of \emph{Borel type} if it satisfies $(*)$. We mention that this concept appears also in \cite[Definition 1.3]{CS} as the so called {\em weakly stable ideal}.

Let $I\subset S$ be a monomial ideal and $M$ a multigraded $S$-module. We denote $\Gamma_{I}(M)= \bigcup_{k\geq 1} (0:_M I^k)$. In this paper, we extend the concept of Borel type ideals in the frame of multigraded $S$-modules. In order to do that, we note that if $I,J\subset S$ are monomial ideals, then $\Gamma_{J}(S/I) = (I:J^{\infty})/I$. Therefore, the condition $(*)$ can be rewritten as $\Gamma_{(x_j)}(S/I) = \Gamma_{(x_1,\ldots,x_j)}(S/I)$, for all $j=1,\ldots,n$.

For any $S$-module $M$, we denote $t(M)=\{ z\in M:\; (\exists)u\in S $ with $u\cdot z = 0\}$ the torsion submodule of $M$. We say that a \Zng $M$ with $t(M)=M$ is of Borel type if $\Gamma_{(x_j)}(M) = \Gamma_{(x_1,\ldots,x_j)}(M)$ for all $j\in [n]$, see Definition $1.1$. We prove that Borel type modules are sequential Cohen-Macaulay, see Theorem $1.11$. Also, we prove that Borel type modules are pretty clean, see Theorem $1.14$.

In the second part of the paper, we study the regularity of Borel type modules. We give a formula for the regularity of
modules of Borel type, similar with the formula given by Popescu and Herzog in \cite{hp} for the regularity of ideals of Borel type, see Theorem $2.4$. Also, in Proposition $2.7$, we give a new characterization for modules of Borel type.

\vspace{20mm} \noindent {\footnotesize
\begin{minipage}[b]{15cm}
 Mircea Cimpoeas, Simion Stoilow Institute of Mathematics of the Romanian Academy\\
 E-mail: mircea.cimpoeas@imar.ro
\end{minipage}}

\newpage
\section{Multigraded modules of Borel type.}

\begin{dfn}
Let $M$ be a \Zng with $t(M)=M$. We say that $M$ is a module of \emph{Borel type}, if 
$ \Gamma_{(x_i)}(M) = \Gamma_{(x_1,\ldots,x_i)}(M),\;\;(\forall)i\in [n]$.
\end{dfn}

\begin{prop}
Let $M$ be a \Zng with $t(M)=M$. The following are equivalent:

(1) $M$ is of Borel type.

(2) $\Gamma_{(x_i)}(M)\subset \Gamma_{(x_j)}(M)$ for all $1\leq j<i\leq n$.

(3) Any prime $P\in Ass(M)$ is of the form $P=(x_1,\ldots,x_r)$ for some $r\in [n]$.
\end{prop}

\begin{proof}
$(1)\Rightarrow(2)$. Since $j<i$, it follows that 
$\Gamma_{(x_i)}(M) = \Gamma_{(x_1,\ldots,x_i)}(M)\subset \Gamma_{(x_1,\ldots,x_j)}(M) =  \Gamma_{(x_j)}(M)$, as required.
$(2)\Rightarrow(1)$. We prove by induction on $1\leq i\leq n$ that $\Gamma_{(x_i)}(M) = \Gamma_{(x_1,\ldots,x_i)}(M)$. For $i=1$, the assertion is trivial. Suppose $i<n$ and $\Gamma_{(x_i)}(M) = \Gamma_{(x_1,\ldots,x_i)}(M)$. Since $\Gamma_{(x_{i+1})}(M)\subset \Gamma_{(x_i)}(M)$, it follows that $\Gamma_{(x_{i+1})}(M)\subset \Gamma_{(x_1,\ldots,x_i)}(M)$. Let $z\in \Gamma_{(x_{i+1})}(M)$. It follows that there exist a nonnegative integer $k$,
such that $x_{i+1}^k\cdot z = 0$. Since $\Gamma_{(x_{i+1})}(M)\subset \Gamma_{(x_1,\ldots,x_i)}(M)$, it follows that
there exists a nonnegative integer $k'$ such that $(x_1,\ldots,x_i)^{k'}\cdot z = (0)$. Thus, $(x_1,\ldots,x_i,x_{i+1})^{k+k'}\cdot z = (0)$ and therefore $z\in \Gamma_{(x_1,\ldots,x_{i+1})}(M)$. We obtain $\Gamma_{(x_{i+1})}(M) \subset \Gamma_{(x_1,\ldots,x_i,x_{i+1})}(M)$ and thus we complete the proof.

$(2)\Rightarrow(3)$. Let $P\in Ass(M)$ be a prime ideal. Since $M$ is \Zng, it follows that $P$ is a monomial prime ideal, i.e. $P$ is generated by a set of variables. Since $P\neq (0)$, there exists a $x_i\in P$ for some $i\in [n]$. We claim that $x_j\in P$ for all $j\leq i$. Assume $P=(0:_S z)$, for some $z\in M$. It follows that $x_i z = 0$. Let $j<i$. Since $\Gamma_{(x_i)}(M)\subset \Gamma_{(x_j)}(M)$, it follows that $x_j^k z = 0$ for some $k\gg 0$ and therefore $x_j^k \in P$. Thus $x_j\in P$, as required.

$(3)\Rightarrow (1)$. Let $0\neq z\in \Gamma_{(x_i)}(M)$. It follows that $x_i^k z = 0$ for some $k\gg 0$. Let $I=(0:_S z)$ and $P\supset I$, $P\in Ass(M)$. Since $x_i\in P$, it follows that $(x_1,\ldots,x_i)\subset P$. Therefore, $(x_1,\ldots,x_i)^{\infty}z=0$ and thus $\Gamma_{(x_i)}(M) \subset \Gamma_{(x_1,\ldots,x_i)}(M)$.
\end{proof}

\begin{cor}
Let $M$ be a \Zng with $t(M)=M$. Let $N\subset M$ be a submodule. Then:

(1) If $M$ is of Borel type then $N$ is of Borel type.

(2) If $N$ and $M/N$ are of Borel type then $M$ is of Borel type.
\end{cor}

\begin{proof}
We use Proposition $1.2(3)$ and the well known inclusions $Ass(N)\subset Ass(M) \subset Ass(N)\cup Ass(M/N)$.
\end{proof}

\begin{cor}
(1) If $I\subset S$ is a monomial ideal then $S/I$ is of Borel type, if and only if $I$ is an ideal of Borel type.

(2) If $J\subset I\subset S$ are two monomial ideals, and $J$ is of Borel type, then $I/J$ is of Borel type. 
    Conversely, if $I/J$ is of Borel type, then $J$ is of Borel type if and only if $I$ is of Borel type.
\end{cor}

\begin{proof}
(1) We have $\Gamma_{(x_i)}(S/I) = (I:x_i^{\infty})/I$ and $\Gamma_{(x_1,\ldots,x_i)}(S/I) = (I:(x_1,\ldots,x_i)^{\infty})/I$, for any $i\in [n]$. Therefore, $\Gamma_{(x_i)}(S/I) = \Gamma_{(x_1,\ldots,x_i)}(S/I)$ if and only if $(I:x_i^{\infty}) = \linebreak (I:(x_1,\ldots,x_i)^{\infty})$.
(2) We consider the short exact sequence $0 \rightarrow I/J \rightarrow S/J \rightarrow S/I \rightarrow 0$ and we apply  Corollary $1.3, (1)$ and $(2)$.
\end{proof}

\begin{lema}
Let $M$ be a \Zng of Borel type. Then, for any integers $j,p$ with $1\leq j \leq n$ and $0 \leq p \leq n-j$, we have
\[ \Gamma_{(x_j)}(M) = \Gamma_{(x_j \cdots x_{j+p})}(M). \]
\end{lema}

\begin{proof}
We use induction on $p$. If $p=0$, there is nothing to prove. Assume $p>0$, $j<n$ and 
$\Gamma_{(x_j)}(M) = \Gamma_{(x_j \cdots x_{j+p-1})}(M)$. Let $z\in \Gamma_{(x_j \cdots x_{j+p})}(M)$ be a homogeneous element. If follows that there exists a positive integer $k$ such that $(x_j \cdots x_{j+p})^k \cdot z = 0$. We have 
$(x_j \cdots x_{j+p})^k  z = x_{j+p}^k (x_j \cdots x_{j+p-1})^k z = 0$ and therefore $(x_j \cdots x_{j+p-1})^k z \in \Gamma_{(x_{j+p})}(M) \subseteq \Gamma_{(x_{j})}(M)$. Thus, there exists a positive integer $k'$ such that
$x_j^{k'} (x_j \cdots x_{j+p-1})^k z = 0$. Moreover, $(x_j \cdots x_{j+p-1})^{k+k'} z = 0$, and thus $z\in \Gamma_{(x_j \cdots x_{j+p-1})}(M)$. Since $\Gamma_{(x_j \cdots x_{j+p-1})}(M) \subseteq \Gamma_{(x_j \cdots x_{j+p})}(M)$, we get $\Gamma_{(x_j \cdots x_{j+p-1})}(M) = \Gamma_{(x_j \cdots x_{j+p})}(M)$.
\end{proof}

\begin{prop}
Let $M$ be a \Zng of Borel type.

(1) Let $u\in S$ be a monomial. Denote $w(u)=\min \{i:\; x_i|u\;\}$. Then $\Gamma_{(u)}(M) = \Gamma_{(x_{w(u)})}(M)$.

(2) $t(M)=\Gamma_{(x_1)}(M) = \Gamma_{(x_1\cdots x_n)}(M) = M$.
\end{prop}

\begin{proof}
(1) It is well known that $\Gamma_{(u)}(M) = \Gamma_{(\sqrt{u})}(M)$, where $\sqrt{u}=\prod_{x_i|u} x_i$. By previous lemma, we have $\Gamma_{(u)}(M) = \Gamma_{(\sqrt{u})}(M) \subseteq \Gamma_{(x_{w(u)}\cdots x_n)}(M) = \Gamma_{(x_{w(u)})}(M) \subseteq \Gamma_{(u)}(M)$ and thus $\Gamma_{(u)}(M) = \Gamma_{(x_{w(u)})}(M)$ as required.

(2) Let $z\in M$ be a homogeneous element. It follows that there exists a monomial $u\in S$ such that $uz=0$. Therefore, by $(1)$, we get $z \in \Gamma_{(u)}(M) = \Gamma_{(x_{w(u)})}(M)$, where $w(u)=\min \{i:\; x_i|u\;\}$. Since $M$ is of Borel type, it follows that $z\in \Gamma_{(x_1)}(M)$. This complete the proof.
\end{proof}

Let $M$ be a \Zng of Borel type. We consider the sequence
\begin{equation}
0\subset \Gamma_{(x_n)}(M) \subset \Gamma_{(x_{n-1})}(M) \subset \cdots \subset \Gamma_{(x_1)}(M) = M.
\end{equation}
Note that, in this sequence, we may have some equalities. We denote $M_0:=0$. Suppose $M_{\ell-1}$ is defined for some integer $\ell>0$. Let $n_{\ell}:=\max\{j| \Gamma_{(x_j)}(M/M_{\ell-1})\neq 0 \}$. Let $M_{\ell}:=\Gamma_{(x_{n_{\ell}})}(M)$. 
We claim that $n_{\ell}<n_{\ell-1}$. Indeed, if $j\geq n_{\ell-1}$, then $\Gamma_{(x_j)}(\Gamma_{x_{n_{\ell-1}}}(M_{\ell-1}))=\Gamma_{(x_j)}(M_{\ell-1})$. Therefore, $\Gamma_{(x_j)}(M/\Gamma_{(x_{n_{\ell-1}})}(M))=0$. We obtained a sequence of multigraded modules, 
\begin{equation}
0 = M_0 \subset M_1 \subset \cdots \subset M_r = M , 
\end{equation} 
which remove the equalities from $(1)$, and we called the \emph{sequential chain of modules} of $M$.

Note that, if $M=S/I$, the sequential chain of $M$ is given by $M_{\ell} = I_{\ell}$, where 
$0 = I_0 \subset I_1 \subset \cdots \subset I_r = S$ is the sequential chain of $I$, as been defined in $\cite{hpv}$.

\begin{dfn}
Let $M$ be a finitely generated graded $S$-module. The module $M$ is sequentially Cohen-Macaulay, if there exists a finite filtration
$0 = M_0 \subset M_1 \cdots \subset M_r = M$
of $M$ by graded submodules of $M$ such that each quotient $M_i/M_{i-1}$ is Cohen Macaulay and $dim(M_1/M_0)< dim(M_2/M_1) < \cdots < dim(M_r/M_{r-1})$. We call if the \emph{$CM$-filtration} of $M$.
\end{dfn}

Let $M$ be a graded $S$-module with $dim(M)=d$. We denote $D_i(M)$ the largest submodule of $M$ with $dim(D_i(M))\leq i$, where $i\in[d]$. The filtration $0\subset D_0(M)\subset D_1(M)\subset \cdots \subset D_d(M)=M$ is called the \emph{dimension filtration} of $M$. We recall the following results of Schenzel, see \cite{sch}.

\begin{prop}(Schenzel)
(1) $M$ is sequentially Cohen Macaulay, if and only if the dimension filtration of $M$ is a $CM$-filtration.

(2) Let $(0)=\bigcap_{j=1}^m Q_j$ be a primary decomposition of $(0)$, where $Q_j$ is $P_j$-primary. Then:
\[ D_i(M) = \bigcap_{dim(S/P_j)>i} Q_j. \]
\end{prop}

\begin{prop}
Let $M$ be a \Zng of Borel type. We consider the sequential chain of $M$. Then $D_{n-i}(M)=\Gamma_{(x_i)}(M)$.
\end{prop}

\begin{proof}
Assume $Ass(M)=\{P_1,\ldots,P_r\}$, where $P_1\subset P_2\subset \cdots \subset P_r$. Since $M$ is of Borel type, it follows that
$P_j=(x_1,\ldots,x_{i_j})$, where $1\leq i_1<i_2<\cdots<i_r=n$. We use induction on $r\geq 1$. If $r=1$, then $M$ is $P_1=(x_1,\ldots,x_{i_1})$-coprimary and therefore, $\Gamma_{(x_{i_1})}(M) =\Gamma_{P_1}(M) = M$. Also, $\Gamma_{(x_j)}(M)=0$ for $j>i_1$, since $x_j$ is regular on $M$, and $\Gamma_{(x_j)}(M)=M$ for $j\leq i_1$, since $\Gamma_{(x_j)}(M)\subset \Gamma_{(x_{i_1})}(M)$. Since $dim(M)=n-i_1$, by previous Proposition, we are done.

Assume $r>1$. We consider $M_1=\Gamma_{(x_{i_1+1})}(M)$. We claim that $P_1\notin Ass(M_1)$. Indeed, $(P_1,x_{i_1+1})^k \subset Ann(M_1)$ for some $k\gg 1$. But any $P\in Ass(M_1)$ contains $Ann(M)$, and thus $(x_1,\ldots,x_{i_1},x_{i_1+1})\subset P$. It follows that $Ass(M_1)\subset \{P_2,\ldots,P_r\}$. Using the induction hypothesis, it follows that $D_{n-i}(M_1)=\Gamma_{(x_i)}(M_1)$ for all $i>i_1$. Therefore, $D_{n-i}(D_{n-i_1-1}(M)) = D_{n-i}(M) = \Gamma_{(x_i)}(\Gamma_{(x_{i_1})}(M)) = \Gamma_{(x_i)}(M)$, as required.
\end{proof}

\begin{prop}
Let $M$ be a \Zng with $\Gamma_{(x_1)}(M)=M$ and consider the sequential chain $(2)$ of $M$. The following are equivalent:

(1) $M$ is of Borel type.

(2) $x_{n_{\ell}+1},\ldots,x_{n}$ is a regular sequence on $Q_{\ell}=M_{\ell}/M_{\ell-1}$.

(3) $H_{(x_{n_{\ell}+1},\ldots,x_{n})}^i (Q_{\ell}) = 0$ for all $i<n-n_{\ell}$.
\end{prop}

\begin{proof}
For $(1)\Leftrightarrow (2)$, note that $\bar{Q}_{\ell}:=Q_{\ell}/((x_{n_{\ell}+1},\ldots,x_{n})Q_{\ell})$ has a structure of an Artinian multigraded $S_{\ell}=K[x_1,\ldots,x_{\ell}]$-module. Indeed, $\Gamma_{(x_1,\ldots,x_{\ell})S_{\ell}}(\bar{Q}_{\ell}) = \bar{Q}_{\ell}$, since $\Gamma_{(x_1,\ldots,x_{\ell})}(Q_{\ell}) = \Gamma_{(x_{\ell})}(Q_{\ell}) = \Gamma_{(x_{\ell})}(M_{\ell})/ \Gamma_{(x_{\ell})}(M_{\ell-1}) = M_{\ell}/M_{\ell-1}= Q_{\ell}$. Since $dim(Q_{\ell})=n-n_{\ell}$, it follows that $x_{n_{\ell}+1},\ldots,x_{n}$ is a system of parameters and, thus a regular sequence on $Q_{\ell}$. 

$(2)\Leftrightarrow (3)$ is a well known result of local cohomology.
\end{proof}

\begin{teor}
Let $M$ be a \Zng of Borel type. Then $M$ is sequentially Cohen-Macaulay.
\end{teor}

\begin{proof}
We consider the sequential chain of $M$. Then, according to the previous Proposition, the sequential chain of $M$ is the $CM$-filtration of $M$ and $dim(M_{\ell-1}/M_{\ell})=n-n_{\ell}$.
\end{proof}

Let $M$ be a finitely generated graded $S$-module. A prime filtration of $M$ is a filtration
\[\mathcal F:\; 0 = M_0 \subset M_1 \subset \cdots \subset M_r = M\]
such that $M_i/M_{i-1}\cong S/P_i$ for some $P_i\subset S$ prime. We denote $supp(\mathcal F) = \{P_1,\ldots,P_r\}$. It is well known, 
see \cite[Theoreom 6.5]{mats}, that $Ass(M)\subset Supp(\mathcal F)\subset Supp(M)$. Dress \cite{dress} calls a prime filtration $\mathcal F$ of $M$ \emph{clean}, if $supp(\mathcal F) = Min(M)$. The module $M$ is called \emph{clean}, if $M$ admits a clean filtration.

We recall the following definition, which generalize cleanness, introduced by Herzog and Popescu in \cite{hep}.

\begin{dfn}
We say that $M$ is pretty clean, if there exists a prime filtration, $0 = M_0 \subset M_1 \subset \cdots \subset M_r = M$
of $M$, whit $M_i/M_{i-1}\cong S/P_i$ for some prime ideal $P_i\subset S$, such that if $i<j$ and $P_i\subset P_j$, then $P_i=P_j$.
\end{dfn}

We denote $supp(\mathcal F)=\{P_1,\ldots,P_r\}$. It is well known, see \cite[Corollary 3.6]{hep}, that $Ass(M)=Supp(\mathcal F)$ if $M$ is pretty clean. 

\begin{lema}
Let $M$ be a \Zng. Suppose $x_n$ is regular on $M$. Then $M$ is clean if and only if $M/x_n M$ is a clean $S'=K[x_1,\ldots,x_{n-1}]$-module.
\end{lema}

\begin{proof}
Assume $M$ is clean and take $0 = M_0 \subset M_1 \subset \cdots \subset M_r = M$ a clean prime filtration of $M$. Then,
$0 = M_0/x_n M_0 \subset M_1/x_nM_1 \subset \cdots \subset M_r/x_nM_r = M/x_n M$
is a clean prime filtration of $M/x_nM$, as an $S'$-module. Indeed, if $M_i/M_{i-1}\cong S/P$, where $P\in Min(M)$, then $(M_i/x_n M_i) / (M_{i-1}/x_n M_{i-1}) \cong S'/P'$, where $P'=P\cap S \in Min(M')$.

For the converse, assume $M/x_nM$ is clean and take $0 = M_0/x_n M \subset M_1/x_nM \subset \cdots \subset M_r/x_n M = M/x_n M$
a clean prime filtration of $M/x_n M$ as an $S'$-module. This filtration is also a prime filtration of $M/x_n M$ as a $S$-module. In this regard, we have $(M_i/x_n M_i) / (M_{i-1}/x_n M_{i-1}) \cong M_i/M_{i-1} \cong S/P$, where $P\in Min(M)$, since $P'=P\cap S'\in Min(M')$. Therefore, $0 = M_0 \subset M_1 \subset \cdots \subset M_r = M$ is a clean prime filtration of $M$.
\end{proof}

\begin{teor}
Let $M$ be a \Zng of Borel type. Then $M$ is pretty clean.
\end{teor}

\begin{proof}
Let $Ass(M)=\{P_1,\ldots,P_r\}$ and suppose $P_1\supset P_2\supset \cdots \supset P_r$. The submodule $(0)\subset M$ can be written as 
$(0)=\bigcap_{i=1}^r Q_i$, where $Q_i$ is $P_i$-primary. Set $U_i=\bigcap_{j>i}Q_j$. According to 1.6(2), $D_{d_i}(M)=U_i$, where $d_i=dim(S/P_i)$.

According to \cite[Corollary 4.3(c)]{hep} it is enough to prove that $U_i/Q_i\cap U_i$ is clean for all $i$. We have $\emptyset\neq Ass(U_i/Q_i\cap U_i) = Ass((U_i+ Q_i)/Q_i) \subset Ass(M/Q_i)=\{P_i\}$. Let $J_i=\{x_j|\;x_j\in P_j\}$ and $S'_i=K[J_j]$. Then $P'_i:=P_i\cap S_i$ is the graded maximal ideal of $S'_i$. Assume $P_i=(x_1,\ldots,x_r)$ and let $Q'_i = Q_i/(x_{r+1},\ldots,x_n)Q_i$ and $U'_i=U_i/(x_{r+1},\ldots,x_n)U_i$. The $S'_i$-module $U'_i/Q'_i\cap  U'_i$ is clean, since it is of finite length. We claim that $U_i/U_i\cap Q_i$ is also clean. Indeed, $x_{r+1},\ldots,x_n$ is a regular sequence on $Q_k$ for all $k\geq i$ and therefore, $x_{r+1},\ldots,x_n$ is a regular sequence of $U_i$. Now, since 
\[ (U_i/U_i\cap Q_i) /((x_{r+1},\ldots,x_n)(U_i/U_i\cap Q_i)) \cong U'_i/U'_i\cap Q'_i, \]
it follows by Lemma $1.8$ that $U_i/U_i\cap Q_i$ is also clean.
\end{proof}





\newpage
\section{Regularity of Borel type modules.}

\begin{dfn}
Let $M$ be a graded $S$-module. The \emph{Castelnuovo-Mumford regularity} $reg(M)$ of $M$ is
\[ \max_{i,j} \{j-i :\; \beta_{ij}(M)\neq 0\}.\]
\end{dfn}

Let $\{a_i(M)=sup\{j:H_{\mathbf m}^i(M)_j\neq 0\}$. It is well known, that $reg(M)=\max \{a_i(M)+i:\;i\geq 0\}$. Assume $M$ is sequentially Cohen-Macaulay with the $CM$-filtration $0 = M_0 \subset M_1 \cdots \subset M_r = M$. We set $d_i=dim(M_i/M_{i-1})$. We have the following obvious equations:

\begin{lema}
(a) $dim(M)=d_r$ and $depth(M)=d_1$;

(b) $a_{d_i}(M)= a_{d_i}(M_i/M_{i-1})$ for $i\in [r]$ and $a_j(M)=- \infty$ for $j\notin \{d_1,\ldots,d_r\}$. In particular, if follows that \[ reg(M) = \max\{reg(M_i/M_{i-1}):\;i\in[r] \} = \max\{a_{d_i}(M_i/M_{i-1}) + d_i:\;i\in [r] \}. \]
\end{lema}

Let $\omega_S$ denote the canonical module of $S$. Sequentially $CM$-modules over $S$ can be characterized homologically as follows.

\begin{teor}
The following conditions are equivalent:

(a) $M$ is sequentially $CM$.

(b) for all $0\leq i\leq dim(M)$, the module $Ext^{n-i}_S(M,\omega_S)$ is either $0$ or Cohen-Macaulay of dimension $i$.

If the equivalent conditions hold and $0 = M_0 \subset M_1 \cdots \subset M_r = M$ is the $CM$-filtrations of $M$ with $d_i=dim(M_i/M_{i-1})$, then
\[ Ext_S^{n-d_i}(M,\omega_S) = Ext_R^{n-d_i}(M_i/M_{i-1},\omega_S)\;\;for\;\;i\in [r]  \]
and $Ext_S^{j}(M,\omega_S) = 0$ for $j\notin \{n-d_1,\ldots,n-d_r\}$.
\end{teor}

If $N$ is an Artinian $S$-module, we denote $s(N)=\max\{j:\; N_j\neq 0\}$.

\begin{teor}
Let $M$ be a \Zng of Borel type. Using the notations of Proposition $1.10$, we have:
\[ reg(M)=\max\{s(\bar{Q}_{1}),\ldots,s(\bar{Q}_{r})\}.\]
\end{teor}

\begin{proof}
By Thereom $2.3$, we have $Ext_S^j(M,S(-n)) = 0$ for $j\notin \{n_0,\ldots,n_{r-1}\}$. 
\[ Ext^{n_i}_S(M,S(-n)) \cong Ext^{n_i}_S(M_{i}/M_{i-1},S(-n)),\;\;for\;i\in [r]. \]
Since $x_{n_i+1},\ldots,x_n$ is a regular sequence on $Q_i=M_{i}/M_{i-1}$, 
it follows that \[ s(Ext^{n_i}_{S_i}(\bar{Q_i}, S_i(-n_i))) = s(Ext^{n_i}_S(Q_i, S(-n))) + n - n_i,\] where $\bar{Q}_i=Q_i/ (x_{n_i+1},\ldots,x_n)Q_i$.
On the other hand, using local duality, we have $a_i(M)=sup\{j:\; Ext_S^{n-i}(M,S)_{-n-j}\neq 0\}$. Therefore, 
\[a_{n-n_i}(M) = a_{n-n_i}(M_i/M_{i-1}) = s(\bar{Q_j}) - n + n_i \]
and thus $reg(M)=\max\{a_{n-n_i}(M)+n-n_i:\;i\in[r]\} =\max\{s(\bar{Q}_{1}),\ldots,s(\bar{Q}_{r})\}$, as required.
\end{proof}

\begin{dfn}
Let $M$ be a \Zng with $t(M)=M$. We say that $M$ is \emph{strongly stable}, if $(0:_M x_i)\subset (0:_M x_j)$ 
for any $1\leq j<i\leq n$.
\end{dfn}

\begin{obs}
If $I\subset S$ is a monomial ideal and $M=S/I$, then $M$ is strongly stable if and only if $I$ is strongly stable.
\end{obs}

Let $M$ be a \Zng. Note that $M$ has an induced structure of $\mathbb Z$-graded $S$-module, given by 
$ M = \bigoplus_{j \in \mathbb Z} \left( \bigoplus_{|a|=j,\;a\in \mathbb Z^n} M_a \right)$ , where 
\linebreak $|a|=a_1+\cdots + a_n$.

If $e$ is an integer, we denote $M_{\geq e}$ the submodule of $M$ generated by the set \linebreak $\{z\in M$ homogeneous $:\;deg(z)\geq e\}$.

\begin{prop}
Let $M$ be a \Zng. If $M_{\geq e}$ is strongly stable for some $e\gg 0$ then $M$ is of Borel type.
\end{prop}

\begin{proof}
Let $1\leq j<i\leq n$ and let $z\in (0:_M x_i^{\infty})$ be a homogeneous element. It follows that $x_i^k z=0$ for some $k\gg 0$. We may assume $deg(z)+k-1\geq e$. Note that $x_i^{k-1}z \in (0:_{M_{\geq e}} x_i)$. It follows that $x_i^{k-1}z \in (0:_{M_{\geq e}} x_j)$, since $M_{\geq e}$ is strongly stable. Therefore $x_j x_i^{k-1}z = 0$ and thus $x_jx_i^{k-2}z\in (0:_{M_{\geq e}} x_n)$. Using the same procedure, finally, we get $x_j^{k-1}z\in (0:_{M\geq e} x_j)$ and therefore $x_j^k z =0$, thus $z\in (0:_M x_j^{\infty})$. By Proposition $1.2(2)$ we are done.
\end{proof}

\end{document}